\documentclass[letterpaper, 10 pt, conference]{ieeeconf}
\IEEEoverridecommandlockouts

\let\labelindent\relax
\usepackage{comment}
\usepackage{booktabs}
\usepackage{cite}
\usepackage{mathtools}
\usepackage{amsmath,amssymb,amsfonts}
\usepackage{algorithmic}
\usepackage{graphicx}
\usepackage{textcomp}
\usepackage{makecell}
\usepackage{xcolor}
\usepackage{amsthm}
\usepackage{microtype}
\usepackage{enumitem}
\newtheorem{definition}{Definition}

\newcommand{\R}{\mathbb{R}}

\newcommand{\ip}[2]{\left\langle #1,#2\right\rangle}
\usepackage{amsmath}
\DeclareMathOperator*{\esssup}{ess\,sup}

\newcommand{\mcl}[1]{\mathcal{#1}}
\newcommand{\mbb}[1]{\mathbb{#1}}
\newcommand{\mbf}[1]{\mathbf{#1}}

\newcommand{\norm}[1]{\left\lVert #1 \right\rVert}

\newcommand{\mat}[1]{\begin{matrix} #1 \end{matrix}}

\newcommand{\bmat}[1]{\begin{bmatrix} #1 \end{bmatrix}}
\newtheorem{theorem}{Theorem} 
\newtheorem{corollary}{Corollary}
\theoremstyle{definition}
\newtheorem{example}{Example}
\def\BibTeX{{\rm B\kern-.05em{\sc i\kern-.025em b}\kern-.08em
    T\kern-.1667em\lower.7ex\hbox{E}\kern-.125emX}}
\allowdisplaybreaks
\begin{document}

\title{\LARGE \bf Impulse-to-Peak-Output Norm Optimal State-Feedback Control of Linear PDEs}%
\author{Tristan Thomas$^{1}$, Sachin Shivakumar$^{2}$, and Javad Mohammadpour Velni$^{1}$
\thanks{$^{1}$Tristan Thomas \{{\tt\small tthoma7@clemson.edu}\} and Javad Mohammadpour Velni \{{\tt\small javadm@clemson.edu}\} are with the Department of Mechanical Engineering at Clemson University.}%
\thanks{$^{2}$Sachin Shivakumar \{{\tt\small sshivakumar@lanl.gov}\} is with the Center for Nonlinear Studies, Los Alamos National Laboratory. }%
\thanks{This work is partly supported by the LDRD program (20250614CR-NLS) of Los Alamos National Laboratory. LA-UR-26-22329. Tristan Thomas is supported by the U.S. Department of Education through GAANN program P200A220075.}
}


\maketitle

\begin{abstract}
Impulse-to-peak response (I2P) analysis for state-space ordinary differential equation (ODE) systems is a well-studied classical problem. However, the computational techniques employed for I2P optimal control of ODEs have not been extended to partial differential equation (PDE) systems due to the lack of a computation-friendly universal transfer function and state-space representation. Recently, however, partial integral equation (PIE) representation was proposed as the desired state-space representation of a PDE, and Lyapunov stability theory was used to solve various control problems, such as stability and $H_\infty$-optimal control. In this work, we utilize this PIE framework, and associated Lyapunov techniques, to formulate the I2P response analysis problem as a solvable convex optimization and obtain provable bounds for the I2P-norm of linear PDEs. Moreover, by establishing strong duality between primal and dual formulations of the optimization problem, we develop a constructive method for I2P optimal state-feedback control of PDEs and demonstrate the effectiveness of the method on various examples.

\end{abstract}


\section{Introduction}

Optimal control of partial differential equations (PDEs) has been a well-studied topic \cite{Moura,vanKeulen,troltzsch2010optimal} owing to their ubiquitous use in a range of practical applications. However, metrics used therein cannot certify safety-critical transient behaviors. For instance, instantaneous spikes in current or temperatures that often lead to failure/degradation in Lithium-ion batteries are not captured by $H_\infty$/LQR metrics~\cite{mouraCST}. Hence, one often uses optimal control approaches such as MPC \cite{Guay} to enforce such transient specifications. Although widely used in ordinary differential equation (ODE) systems, such predictive methods require forward simulation of the system dynamics -- a computationally expensive step for PDEs. Since the goal is to design a computationally inexpensive control that can certify transient behavior, unlike $H_\infty$/LQR-optimal control, impulse-to-peak (I2P) norm optimal control is a natural candidate. Hence, we develop a computational method to design I2P optimal state-feedback controllers for a large class of PDEs.  

\enlargethispage{\baselineskip}
Due to simplicity and maturity of tools available for ODE optimal control, the standard approach for I2P optimal control of PDEs often involves approximating a PDE by an ODE (\emph{early-lumping}) and then using ODE methods, such as \cite{TOKUNAGA19981473}, to design an I2P-optimal control. However, approximation methods often introduce truncation errors---resulting in spillover effects and a lack of closed-loop stability or performance bounds \cite{balas}. \emph{Late-lumping} approaches, on the other hand, use variational calculus, Hamilton-Jacobi-Bellman (HJB) equation, or a Riccati PDE and can provide provable performance \cite{Lasiecka_Triggiani_2000}, but the controllers must be obtained by solving nonlinear PDEs with unbounded input operators, which cannot be solved exactly in all but the simplest systems. While \cite{mironchenko2020input} developed Lyapunov characterization of this problem, and sidesteps the above limitations, it does not provide a constructive method to find controllers. Thus, existing approaches typically lack \textit{provable} performance, \textit{generality}, or both.

Naturally, to develop a general algorithm that computes controllers with provable performance, one must first separate the method from the model. A traditional example of this is linear matrix inequality (LMI)-based methods for ODE analysis/control, where a single synthesis method works for all ODE systems of a particular class. Unfortunately, a similar method is not available for PDEs due to lack of a general representation that admits such finite-time algorithms. While the semigroup approach \cite{curtain2012introduction} does provide a general state-space representation for PDEs, it is not convenient for LMI-style computational methods. To overcome this, we use the partial integral equation (PIE) representation \cite{shivakumar_extension_2024}---an \textit{equivalent} state-space-like representation for linear PDEs parameterized by bounded integral operators called partial integral (PI) operators. Moreover, PI operators are closed under composition, addition, and adjoint. Thus, unlike with PDEs, various LMI-style conditions with PI operator variables can analogously be defined and \emph{solved} for PIEs. 

This framework of using PIEs to represent PDEs, posing PDE optimal control problem as an equivalent PIE optimal control problem, and using associated computational machinery to solve the PIE control problems has been successfully demonstrated for $H_\infty$-optimal control problems of PDEs \cite{shivakumar_dual_2024}. In this work, we will leverage the universality and computational benefits afforded by the PIE framework to solve I2P-optimal control problem. In that regard, the \textit{key contributions} of this work are: 1) characterizing the I2P-norm of a PDE in terms of its PIE representation, 2) proving equivalence of the I2P-norm for a PIE and its dual, 3) posing the open-loop I2P-norm upper bounding problem and the closed-loop I2P-norm minimization problem as \emph{solvable} convex optimization problems. The resulting analysis and synthesis conditions are computationally implemented using PIETOOLS \cite{shivakumar_pietools_2025} and have been demonstrated on various numerical examples.


\section{Notation and Preliminaries}
We will use lowercase font for functions of time and space, e.g., $x(t)$ and $\mbf x(t,s)$, and uppercase calligraphic font $\mcl A$ for operators on functions. Given $\mcl A$, an operator on a Hilbert space, we denote the adjoint operator w.r.t. canonical inner-product by $\mcl A^*$.
Partial derivatives w.r.t. time and space are denoted by $\partial_t$ and $\partial_s$.
$L_2^n[X]$ denotes the space of $\R^n$-valued Lebesgue square-integrable functions on $X$ and is equipped with the canonical inner-product $\ip{\cdot}{\cdot}_{L_2^n[X]}$. For functions on $\R_+$, we define the supremum (peak) norm as $\norm{x}_\infty = \esssup_{t\ge 0} \norm{x(t)}$ and associated Banach space as $L_\infty[\R_+]$. 
Lastly, we abbreviate $\R^m\times L_2^n[X]$ as $\R L_2^{m,n}[X]$, e.g., if $\begin{bmatrix}
    x \\ \mathbf{x}
\end{bmatrix} \in  \mathbb{R}L_2^{m,n}$ then $x \in \mathbb{R}^m$ and $\mathbf{x}\in L_2^n$.

Next, similar to \cite{shivakumar_dual_2024}, we briefly introduce the Partial Integral (PI) operators, Partial Integral Equations (PIEs), and Linear PI Inequalities (LPIs) optimization problems.

\subsection{The Class of Partial Integral Operators}

4-PI operators, a subclass of general PI operators, are used for the parameterization of PIEs just as matrices are used for the parameterization of linear ODEs. 

\begin{definition}
Given $x \in \mathbb{R}^{m_1}$ and $\mathbf{x} \in L_2^{n_1}[a,b]$, we say $\mathcal{P}:\R L_2^{m_1,n_1}\to \R L_2^{m_2,n_2}$ is a \textbf{4-PI operator} if it can be parameterized as \begin{align}
\left(\mcl P \begin{bmatrix}
    x\\ \mathbf{x}
\end{bmatrix}\right)(s) \coloneqq \begin{bmatrix}
    Px + \int_a^b{Q_1(\theta)x(\theta) d \theta }\\
    Q_2(s)x + \mathcal{R}\mathbf {x}(s)\nonumber
\end{bmatrix},\end{align}
where
\begin{align}
    (\mathcal{R}\mathbf{x})(s)\!\!&=\!\!R_0(s)\mathbf{x}(s) \!+\!\! \int_a^s\!\!{R_1(s,\theta)\mbf{x}(\theta)d\theta}\!+\!\!\int_s^b\!\!{R_2(s,\theta)\mathbf{x}(\theta)d\theta}\nonumber
\end{align}
for some matrix $P$ and matrix-valued polynomials $Q_1, Q_2, R_0, R_1,$ and $R_2$. Given $P,Q_i,R_i$, we denote 4-PI operators in the compact form
$\mcl P= \left[ \begin{array}{c|c}
P & Q_1 \\
\hline
Q_2 & \{R_i\}
\end{array} \right]$.
We denote the class of 4-PI operators as $\Pi_4$. 
\end{definition}
The parameter dimensions are inherited from the dimensions of the map, i.e., $P \in \mathbb{R}^{m_2 \times m_1}$, $Q_1(s) \in \mathbb{R}^{m_2 \times n_1}$, $Q_2(s)\in \mathbb{R}^{n_2\times m_1}$, and $R_0(s), R_1(s,\theta), R_2(s,\theta) \in \mathbb{R}^{n_2 \times n_1}$. In the special case $(m_1,n_1) = (m_2,n_2)$, $\Pi_4$ forms a $*$-subalgebra \cite{shivakumar_extension_2024}. Notably, 4-PI operators are always bounded since their kernels are polynomials.  

\subsection{The Class of Partial Integral Equations}
Next, we examine the class of dynamical systems that are parametrized by these 4-PI operators.
\begin{definition}
    Given 4-PI operators $\{\mcl T, \mcl A, \mcl B, \mcl C, \mcl D\}\subset \Pi_4$, we say $\mbf x(t,\cdot)\in \R L_2^{m,n}, z:\R_+ \to \R^{n_z}$ are governed by a {\bf Partial Integral Equation (PIE)} if $\{\mbf x,z\}$ satisfies
    \begin{align}
        \label{eq:PIE}
        \partial_t\left(\mcl T\mbf x\right)(t) &= \mcl A \mbf x(t) \!+\! \mcl B w(t),\quad z(t) = \mcl C \mbf x(t)\!+\!\mcl D w(t),
    \end{align}
    for some initial conditions $\mcl T\mbf x(0)$ and $w(t) \in \mbb{R}^{n_w}$.
\end{definition}
The initial conditions for PIEs are often defined on $\mcl T\mbf x$ instead of $\mbf x$ because, by the nature of PIE representation of PDEs, initial conditions on $\mcl T\mbf x$ are wellposed but not necessarily on $\mbf x$.

Since the 4-PI operators are linear, bounded, and form a *-subalgebra, they are a natural extension of state-space ODE representation to infinite-dimensional spaces. In fact, various linear PDE and time-delay systems \cite{shivakumar_extension_2024} can be described using the form Eq. \eqref{eq:PIE}. For example, consider a transport equation, $\partial_t v(t,s) = \partial_s v(t,s)$ with boundary condition $v(t,0) = 0$. By defining $\mbf{x}(t,s) = \partial_s v(t,s)$ and using the Fundamental Theorem of Calculus, $v(t,s) = \int_0^s \mbf{x}(t,\zeta) \ d\zeta$. This is of the form $\partial_t(\mcl T \mbf{x})(t) = \mcl A\mbf{x}(t)$ where $\mcl T\mbf{x} = \int_0^s \mbf{x}(t,\zeta) \ d\zeta $, and $\mcl A = 1$. We illustrate the class of PDEs considered in this work in the following subsection.


\subsection{The Class of Partial Differential Equations}
Having presented a general class of PIEs, we next present the class of PDEs that admit such a representation. 
For illustration, we consider a class of $2^{nd}$-order PDEs in a single spatial variable on a compact domain, $s \in [a, b]$. 
Given a PDE state $\mbf x = [\mbf x_1, \mbf x_2,\mbf x_3]^\top$, where spatial and time dependence are sometimes omitted for brevity, we can represent such PDEs as 
\begin{align}\label{eq:PDE}
    \partial_t \mbf{x}(s,t) &= \mcl A_p \mbf x(s,t) + B_{21}(s)w(t) + B_{22}(s)u(t),\\
    z(t) &= \mcl C \mbf x(t) + \mcl D_{11} w(t) + \mcl D_{12} u(t),\nonumber
    \end{align}
    $\mbf x_c(s,t) = 
        [\mbf x_2, \mbf x_3, \partial_s\mbf x_{3}]^\top$,
where the operators $\mcl A_p$ and $\mcl C$ are defined as
\begin{align}
    (\mcl A_p \mbf x)(s,t)
    &\coloneqq A_0(s) \mbf x+ A_1(s)[\partial_s\mbf x_{2},
        \partial_s\mbf x_{3}]^\top+A_2(s) \partial_s^2\mbf x_{3} \nonumber,\\
    (\mcl C \mbf x)(t) &\coloneqq \begin{multlined}[t]
         C_{10} [
        \mbf x_c(a,t),\mbf x_c(b,t)]^\top + \int_a^b \Bigl(C_a(s)\mbf{x}\\+C_b(s) 
        [\partial_s\mbf x_{2},\partial_s\mbf x_{3}]^\top\Bigr) ds.  \end{multlined}\nonumber
\end{align}
For all $t\ge 0$, $\mbf x(\cdot, t)$ lies in the domain
   \[D(\mcl A_p) \!\coloneqq\! \left\lbrace \mat{\mbf{x}(s,t) \in L_2^{n_1}[a,b] \times W_{2,1}^{n_2} [a,b] \times W_{2,2}^{n_3}[a,b] \colon \nonumber \\B
        [\mathbf{x}_c(a), \mathbf{x}_c(b)]^\top = 0} \nonumber\right\rbrace,\]
where $B$ defines boundary conditions on the $\mbf x$, and $W_{2,p}^n$ is the canonical Sobolev space, i.e., $\mbf x\in W_{2,p}^n$ if $\partial_s^k\mbf x\in L_2^n$ for all $k\le p$.
If the conditions of \textit{PIE-compatibility} (see \cite[Sec. IV-A]{shivakumar_extension_2024}) are satisfied, any PDE parameterized in the above form can be equivalently represented as a PIE \eqref{eq:PIE}.

\subsection{Solving LPIs and 4-PI Operator Inversion}
Any \textit{PIE-compatible} PDE of the form \eqref{eq:PDE} can be represented as PIEs using explicit formulae in \cite{shivakumar_extension_2024} (or by using the open-source PIETOOLS toolbox \cite{shivakumar_pietools_2025}). Thus, we implicitly assume that the PDE is PIE compatible and the PIE parameters, ($\mcl{T}, \mcl{A}, \mcl{B}, \mcl{C}, \mcl{D}$), are obtained using available tools. In addition, the results herein are presented as linear PI inequality (LPI) optimization problems, which have the form
\begin{align}\label{eq:LPI}
\min_{\mcl P\in \Pi_4}&\quad   f(\mcl P),\quad 
s.t., \quad g_i(\mcl P) = 0,\; h_j(\mcl P)\le 0,
\end{align}
where $f$, $g_i$, and $h_j$ are linear functionals involving PI decision variables $\mcl P$. For example, given a PIE of the form $\partial_t (\mcl T\mbf x)(t) = \mcl A \mbf x(t)$, Lyapunov stability test can be formulated as
\[
\mcl P\succ 0, \;\; \mcl T^*\mcl P\mcl A+\mcl A^*\mcl P\mcl T\preceq 0.
\]
Problems of the above form are solved computationally by parameterizing positive operators using positive matrices. Specifically, one can parameterize $\mcl P = \mcl Z^*P\mcl Z$ using matrix $P>0$ and constrain
\[
\mcl T^*\mcl Z^*P\mcl Z\mcl A+ \mcl A^*\mcl Z^*P\mcl Z\mcl T = -\mcl Z_d^* Q\mcl Z_d,
\]
where $\mcl Z$ and $\mcl Z_d$ are 4-PI operator bases and matrix $Q\ge 0$. Since, 4-PI operator kernels are polynomials, the above constraint is merely a polynomial equality constraint, and thus, is equivalent to constraints on coefficients of polynomials stored in $P$ and $Q$. Thus, infinite-dimensional LPI problem is transformed to finite-dimensional LMI problem and can be solved using interior-point methods. Details of the implementation and procedure of solving LPI constrained problems can be found in \cite{shivakumar_pietools_2025} and \cite{PEET2021109473}.

Lastly, controller synthesis problems often require inversion of 4-PI operators, which can be calculated numerically using the formulae in \cite[Sec. VII]{shivakumar_dual_2024}. The PDE to PIE conversion, LPI optimization, and controller gain reconstruction will be performed using PIETOOLS.


\section{Duality in PIEs}

Every PIE-compatible PDE has an \textit{equivalent} PIE representation, i.e., well-posedness, stability, and $L_2$-gain bound of the PDE implies the same for its PIE representation. However, the properties of the dual system are more critical to controller synthesis problems. Unlike the dual of a PDE, the dual of a PIE is readily available and is defined below.

\begin{definition}\label{def:dualPIE}
Given a PIE Eq. \eqref{eq:PIE} defined by 4-PI operators $\{\mcl T, \mcl A,\mcl B, \mcl C,\mcl D\}\subset \Pi_4$, we say $\bar{\mbf x}:\R_+\to \R L_2^{m,n}$, $\bar z:\R_+\to \R^{n_w}$ are governed by \textbf{the dual PIE}, if $\{\bar{\mbf x},\bar z\}$ satisfies 
\begin{align}
\partial_t(\mcl T^*  \bar{\mbf x})(t)\!=\! \mcl A^* \bar{\mbf x}(t)\!+\!\mcl C^* \bar{w}(t),\; \bar{z}(t)\!=\!\mcl B^*  \bar{\mbf x}(t)\!+\!\mcl D^*\bar w(t), \label{PIEd}
\end{align}
for some initial conditions $\mcl T^*\bar{\mbf x}(0)$ and $\bar w\in L_2^{n_z}[\R_+]$.
\end{definition}
    Since the PI parameters of the primal system, Eq. \eqref{eq:PIE}, lie in a *-subalgebra, the dual system, Eq. \eqref{PIEd}, is also parameterized by PI parameters and admits the same form. Moreover, one can show that a PIE and its dual possess the same stability property and $L_2$-gain bound \cite{shivakumar_dual_2024}.

Both the structural and behavioral equivalence between the primal and dual play a crucial role in the controller synthesis problems. The former allows any computational method developed for the primal to be applicable for the dual. The latter allows one to indirectly establish properties of the primal by proving properties of its dual. For example, the state-feedback controller synthesis for primal is difficult since the Lyapunov stability conditions for the closed-loop system are bilinear (non-convex) in Lyapunov function parameter ($\mcl P$) and controller gains ($\mcl K$): 
\[
\mcl P\succ 0, \; (\mcl A+\mcl B\mcl K)^*\mcl P\mcl T+\mcl T^*\mcl P(\mcl A+\mcl B\mcl K) \preceq 0.
\]
However, since the dual is defined by the same class of parameters and inherits the same stability properties, one can instead find a stabilizing controller by examining Lyapunov conditions for the \textbf{dual} of the closed-loop system, which is a convex problem in decision variables ($\mcl P, \mcl Z=\mcl K\mcl P$):
\[
\mcl P\succ 0, \; (\mcl A\mcl P+\mcl B\mcl Z)\mcl T^*+\mcl T(\mcl P\mcl A^*+\mcl Z^*\mcl B^*) \preceq 0.
\]
By solving the dual constraints, one can find the controller gains $\mcl K = \mcl Z\mcl P^{-1}$ and the associated Lyapunov function parameter $\mcl P$ that proves the closed-loop stability of the primal system. Our objective in this paper is to extend this method for impulse-to-peak norm optimal controller synthesis, and this requires:
\begin{itemize}
    \item finding optimization-based characterization of the impulse-to-peak norm of a PIE Eq. \eqref{eq:PIE}.
    \item showing that impulse-to-peak norm of the primal Eq. \eqref{eq:PIE} and its dual Eq. \eqref{PIEd} are equal.
\end{itemize}



\subsection{Impulse-to-Peak Norm Calculation of a PIE}
First, we will formally define the impulse-to-peak norm of a PIE and establish the equivalence of this system norm in primal and dual PIEs. 
\begin{definition}\label{def:ip-normsystem}
    Given $\{\mcl T, \mcl A, \mcl B, \mcl C\}$, we use $G(\mcl T, \mcl A,\mcl B,\mcl C)$ to denote a PIE with impulse inputs of the form
\begin{align}
\partial_t (\mcl T \mbf x)(t)&= \mcl A \mbf x(t)\!+\!\mcl B w(t),\; \mcl T\mbf x(0) =0,\; z(t)=\mcl C \mbf x(t), \label{PIE2}
\end{align}
where $\mbf x(t) \in \R L_2^{m,n}[a,b]$ is the state, $w(t)$ is an \textbf{impulsive disturbance} of the form $w(t) = \delta(t)v$ with $v \in \mathbb{R}^{n_w}$, and $z(t) \in \R^{n_z}$ is the output.
\end{definition} Then, we can define the impulse-to-peak norm using an auxiliary system where \textit{impulsive disturbance} is replaced by an \textit{initial condition}. 
\begin{definition}\label{def:h2norm}
Given $\{\mcl T, \mcl A, \mcl B, \mcl C\}$, define the auxiliary PIE as 
\begin{align}\label{eqn:PIEaux}
\partial_t (\mcl T \mbf x(t))&= \mcl A \mbf x(t),\quad \mcl T \mbf{x}(0) =\mcl{B} v,\quad z(t)=\mcl C \mbf x(t),
\end{align}
where $v \in \R^{n_w}$. Then, the impulse-to-peak norm of system in Eq.~\eqref{PIE2} is given by
\[
 \norm{G(\mcl T, \mcl A,\mcl B,\mcl C)}_{ip}:=\sup_{\substack{z,\mbf x\, \text{satisfy~\eqref{eqn:PIEaux}}\\ \norm{v}=1}} \norm{z}_{L_\infty}.
\]
\end{definition}
\noindent In simpler words, impulse response of Eq. \eqref{PIE2} is equal to initial condition response of Eq. \eqref{eqn:PIEaux}. This equivalence can be established as follows: Let $w_n(t)\in L_1[\R_+]$ with $w_n(t)\to v\delta(t)$ (converges in a distributional sense). Assuming well-posedness (and by continuity of the input to state map \cite[Lemma 3.1.5]{curtain2012introduction}), we conclude that the sequence of weak solutions of \eqref{PIE2} converges to the weak solution of \eqref{eqn:PIEaux}.

Similar to Def. \ref{def:h2norm}, we can define an auxiliary PIE for the dual of Eq. \eqref{PIE2} as 
\begin{align}\label{eqn:PIEdaux}
\partial_t(\mcl T^* \bar{\mbf x})(t)&= \mcl A^* \bar{\mbf x}(t),\; \mcl T^* \bar{\mbf x}(0) =\mcl{C}^* \bar v,\;\bar z(t)=\mcl B^* \bar{\mbf x}(t),
\end{align}
 and corresponding I2P-norm of $G(\mcl T^*,\mcl A^*,\mcl C^*,\mcl B^*)$ as
\[
 \norm{G(\mcl T^*, \mcl A^*,\mcl C^*,\mcl B^*)}_{ip}:=\sup_{\substack{\bar{z},\bar{\mbf x}\, \text{satisfy~\eqref{eqn:PIEdaux}}\\ \norm{\bar v}=1}} \norm{\bar z}_{L_\infty}.
\]

\subsection{Equivalence in Impulse-to-peak Norm}
Since impulse response of a PIE and initial condition response of its auxiliary system have equivalent I2P-norms, we must now establish that initial condition responses of Eq. \eqref{eqn:PIEaux} and Eq. \eqref{eqn:PIEdaux} are equivalent---implying impulse response of $G(\mcl T,\mcl A,\mcl B,\mcl C)$ and $G(\mcl T^*,\mcl A^*,\mcl C^*,\mcl B^*)$ are equivalent.

\begin{theorem}\label{thm:dualityI2P}
    Given $\mcl{T,A}, \mcl B, \mcl C \in \Pi_4$, if the PIE Eq. \eqref{eqn:PIEaux} and its dual have well-posed solutions, then
    \[
    \norm{G(\mcl T, \mcl A,\mcl B,\mcl C)}_{ip}=\norm{G(\mcl T^*, \mcl A^*,\mcl C^*,\mcl B^*)}_{ip}.
    \]
\end{theorem}
\begin{proof}
    Suppose $\{\mbf x, z\}$ satisfy the PIE, Eq.~\eqref{eqn:PIEaux}, and $\{\bar{\mbf x}, \bar z\}$ satisfy the dual PIE, Eq.~\eqref{eqn:PIEdaux}. As in~\cite{shivakumar_dual_2024}, we can apply integration by parts to get
    \begin{align*}
       I_1:=& \int_0^t \ip{\bar{\mbf x}(t-s)}{\partial_s(\mcl T \mbf x(s))}_{\R L_2}ds
       \\&= \ip{\bar{\mbf x}(0)}{\mcl T \mbf x(t)}_{\R L_2}-\ip{\bar{\mbf x}(t)}{\mcl T \mbf x(0)}_{\R L_2}\\& \underbrace{-\int_0^t \ip{\partial_s\bar{\mbf x}(t-s)}{\mcl T \mbf x(s)}_{\R L_2} ds}_{:=I_2}.
    \end{align*}
Moreover, from the dynamics Eq. \eqref{eqn:PIEaux} and Eq. \eqref{eqn:PIEdaux}, we have
\begin{align*}
    I_1&=\int_0^t \ip{\bar{\mbf x}(t-s)}{\mcl A \mbf x(s)}ds 
    = \int_0^t \ip{\mcl A^*\bar{\mbf x}(t-s)}{\mbf x(s)}d\theta\\
    &= -\int_0^t \ip{\partial_s(\mcl T^*\bar{\mbf x})(t-s)}{\mbf x(s)}d\theta = I_2,
\end{align*}
if $\bar{\mbf x}$ is differentiable.
Therefore, $\ip{\mcl T^*\bar{\mbf x}(0)}{\mbf x(t)}_{\R L_2}=\ip{\bar{\mbf x}(t)}{\mcl T\mbf x(0)}_{\R L_2}$. Next, substituting the initial conditions $\mcl T^*\bar{\mbf x}(0) = \mcl C^* \bar v$ and $\mcl T \mbf x(0)=\mcl B v$, we get
\[
    \ip{\mcl C^* \bar v}{ \mbf x(t)}_{\R L_2}=\ip{\bar{\mbf x}(t)}{\mcl B v}_{\R L_2}.
\]
This implies, for any $t\ge 0$,
\begin{align*}
\ip{\bar v}{ z(t)}_{\R}&= \ip{\bar v}{ \mcl C \mbf x(t)}_{\R}=\ip{\mcl C^*\bar v}{ \mbf x(t)}_{\R L_2}\\
&=\ip{\bar{\mbf x}(t)}{\mcl Bv}_{\R L_2}=\ip{\mcl B^*\bar{\mbf x}(t)}{v}_{\R} =\ip{\bar{z}(t)}{v}_{\R}.
\end{align*}
For clarity, let us make explicit the dependence of $z$ and $\bar z$ on $v$ and $\bar v$, respectively, to rewrite the above identity as $\ip{\bar v}{z(t;v)}_{\R} = \ip{v}{\bar z(t;\bar v)}_{\R}$.

By definition of the vector norm, we have
\begin{align*}
    \sup_{\norm{\bar{v}}=1}\norm{\bar{z}(t;\bar{v})}_{\mathbb{R}} &= \sup_{\norm{\bar v}_2=1,\norm{v}=1}\!\!\!\!\ip{v}{\bar z(t;\bar v)}_{\R}\\
    & = \sup_{\norm{\bar v}_2=1,\norm{v}=1}\!\!\!\!\ip{\bar v}{z(t;v)}_{\R}\!\!=\!\! \sup_{\norm{v}=1}\norm{z(t;v)}_{\mathbb{R}}.
\end{align*}
Since this holds for any $t\in \R_+$, we have 
\begin{align*}
&\sup_{\norm{v}=1} \norm{z}_{\infty} = \sup_{\norm{v}=1}\left(\sup_{t\ge 0} \norm{z(t;v)}_{\mathbb{R}}\right)\\
&\quad =\sup_{t\ge 0} \left(\sup_{\norm{v}=1} \norm{z(t;v)}_{\mathbb{R}}\right)=\sup_{t\ge 0} \left(\sup_{\norm{\bar v}=1} \norm{\bar z(t;\bar v)}_{\mathbb{R}}\right)\\
&\quad =\sup_{\norm{\bar v}=1} \left(\sup_{t\ge 0} \norm{\bar z(t;\bar v)}_{\mathbb{R}}\right)= \sup_{\norm{\bar v}=1} \norm{\bar z}_{\infty}. 
\end{align*}
\end{proof}

\section{Linear PI Inequality Formulations}

Having established the duality result for the impulse-to-peak norm of PIEs, we will next present a convex optimization based characterization of the norm using Lyapunov functions. 
\begin{theorem}\label{thm:Normcoercive}
    Given 4-PI operators ${\mathcal{T},\mathcal{A},\mathcal{B},\mathcal{C}}$, let $\{\mbf x, z\}$ satisfy a PIE of the form
    Eq. \eqref{eqn:PIEaux}. Suppose there exists a $\gamma>0$ and a 4-PI operator $\mcl Q \succeq 0$ such that one of the following two conditions hold:
    \begin{enumerate}[leftmargin=*]
        \item[(i)] $\mathcal{B}^*\mcl Q\mathcal{B}\preceq I$, $\mathcal{A}^*\mcl Q\mathcal{T}+\mathcal{T}^*\mcl Q\mathcal{A} \preceq 0$, $\frac{1}{\gamma^2}\mathcal{C}^*\mathcal{C} \preceq \mathcal{T}^*\mcl Q\mathcal{T}$.
        \item[(ii)] $\mathcal{C}\mcl Q\mathcal{C}^*\preceq I$, $\mathcal{A}\mcl Q\mathcal{T}^*+\mathcal{T}\mcl Q\mathcal{A}^* \preceq 0$, $\frac{1}{\gamma^2}\mathcal{B}\mathcal{B}^* \preceq \mathcal{T}\mcl Q\mathcal{T}^*$.
    \end{enumerate} 
    Then, $\norm{G(\mcl T, \mcl A,\mcl B,\mcl C)}_{ip} \le \gamma$.
\end{theorem}

\begin{proof}[Proof of (i)]
     Let $\{\mbf x, z\}$, satisfy Eq. \eqref{eqn:PIEaux} with initial conditions $\mcl T\mbf x(0) = \mcl Bv$. Define a storage function $V(\mbf x) = \langle\mathcal{T}\mbf x, \mcl Q\mathcal{T}\mbf x\rangle$ where $\mcl Q\succeq 0$ satisfies constraints in \textit{(i)}. Then, $V(0) = \ip{\mcl B v}{\mcl Q\mcl Bv} = \ip{v}{\mcl B^*\mcl Q\mcl B v} \leq \norm{v}^2$. Differentiating $V$ along the solutions of Eq. \eqref{eqn:PIEaux}, we get 
\begin{align*}
&\dot{V} = \langle \partial_t \mathcal{T}\mbf x, \mcl Q \mathcal{T}\mbf x\rangle+\langle \mathcal{T}\mbf x, \mcl Q \partial_t \mathcal{T}\mbf x\rangle\\
&= \langle \mathcal{A}\mbf x, \mcl Q \mathcal{T}\mbf x\rangle\!+\!\langle \mathcal{T}\mbf x, \mcl Q \mathcal{A}\mbf x\rangle \!\! =\!\! \langle \mbf x, (\mathcal{A}^* \mcl Q \mathcal{T}\!+\!\mathcal{T}^*\mcl Q\mathcal{A})\mbf x\rangle\!\le\! 0.    
\end{align*}
Lastly, we have that 
\begin{align*}
\frac{1}{\gamma^2}\norm{z(t)}^2&=\frac{1}{\gamma^2}\ip{\mcl C\mbf x(t)}{\mcl C\mbf x(t)}=\ip{\mbf x(t)}{\frac{1}{\gamma^2}\mcl C^*\mcl C\mbf x(t)}\\
&\le\ip{\mbf x}{\mcl T^*\mcl Q\mcl T\mbf x} = V(\mbf x(t)).
\end{align*}
Thus, $\lVert z(t) \rVert^2 \leq \gamma^2 V(t) \leq \gamma^2 V(0) \leq \gamma^2 \lVert v\rVert^2$ for all $t>0$ and hence, 
\[
\norm{G(\mcl T,\mcl A,\mcl B,\mcl C)}_{ip}=\sup_{\norm{v}=1}\norm{z}_\infty=\sup_{\lVert v\rVert = 1} \sup_{t\geq 0} \lVert z(t) \rVert \leq \gamma.\]
Alternatively, let $\mcl Q$ satisfy the constraints in \textit{(ii)}. Then, following the steps above one can show that $\norm{G(\mcl T^*,\mcl A^*,\mcl C^*,\mcl B^*)}_{ip}\le \gamma$ where $\bar{\mbf x}, \bar z$ satisfy Eq. \eqref{eqn:PIEdaux} with initial conditions $\mcl T^*\bar{\mbf x}(0)=\mcl C^*\bar v$. From Thm. \ref{thm:dualityI2P}, the primal also satisfies the same bound.
\end{proof}


To alleviate conservatism in the sufficient conditions given in Thm \ref{thm:Normcoercive}, we use a broader class of Lyapunov functions to derive the following LPI for upper bounding the norm. 

\begin{theorem}\label{thm:Normnoncoercive}
    Given 4-PI operators ${\mathcal{T}, \mathcal{A},\mathcal{B},\mathcal{C}}$, let $\{\mbf x, z\}$ satisfy a PIE of the form Eq. (\ref{eqn:PIEaux}).
    Suppose there exists a $\gamma > 0$ and a 4-PI operator $\mcl Q$ such that $\mcl T^*\mcl Q =\mcl Q^* \mcl T \succeq 0$ and one of the following conditions hold:
    \begin{enumerate}
        \item[(i)] $\mcl A^*\mcl Q+\mcl Q^*\mcl A \!\preceq\! 0$, $ \begin{bmatrix}
            \gamma^2I & \mcl C\\ \mcl C^* & \mcl Q^* \mcl  T 
        \end{bmatrix}\!\succeq\! 0$, $ \begin{bmatrix}
            \mcl T^*\mcl Q & \mcl Q^* \mcl B\\ \mcl B^* \mcl Q & I
        \end{bmatrix}\!\succeq\! 0$ 
        \item[(ii)] $\mcl A\mcl Q+\mcl Q^*\mcl A^* \!\preceq\! 0$, $ \begin{bmatrix}
            \gamma^2I \!\!& \mcl B^*\\ \mcl B \!\!& \mcl Q^* \mcl  T^*
        \end{bmatrix}\!\succeq\! 0$, $ \begin{bmatrix}
            \mcl T\mcl Q & \mcl Q^* \mcl C^*\\ \mcl C \mcl Q & I
        \end{bmatrix}\!\succeq\! 0$ 
    \end{enumerate} 
    Then, $\norm{G(\mcl T, \mcl A,\mcl B,\mcl C)}_{ip} \le \gamma$.
\end{theorem}

\begin{proof}
We give a proof outline, since it is similar to the proof of Thm. \ref{thm:Normcoercive}.
Let $\{\mbf x, z\}$ satisfy the PIE Eq. \eqref{eqn:PIEaux} with initial conditions $\mcl T\mbf x(0)=\mcl Bv$. Next, define a storage function as $V(\mbf x) = \ip{\mcl T\mbf x}{\mcl Q\mbf x}$. Then, if $\mcl Q$ satisfies the constraints in \textit{(i)}, we have
$\dot{V}(\mbf x(t))\le 0$ and $\norm{z(t)}^2\le \gamma^2V(t)$ for all $t>0$.
        Lastly, for any $v\in \R^{n_w}$ and $\mcl T\mbf x(0)=\mcl Bv$, we have
        \begin{align*}
            0\le &\ip{\bmat{\mbf x(0)\\-v}}{ \bmat{\mcl T^*\mcl Q&\mcl Q^*\mcl B\\\mcl B^*\mcl Q&I} \bmat{\mbf x(0)\\-v}} \\
            &= \ip{\mbf x(0)}{\mcl T^*\mcl Q\mbf x(0)} - 2\ip{\mbf x(0)}{\mcl Q^*\mcl Bv} + \norm{v}^2 \\
            &=\ip{\mbf x(0)}{\mcl Q^*\mcl T\mbf x(0)} - 2\ip{\mbf x(0)}{\mcl Q^*\mcl T\mbf x(0)} + \norm{v}^2 \\
            &= \norm{v}^2-V(\mbf x(0)).
        \end{align*}
        Since $\norm{z(t)}^2\le \gamma^2V(\mbf x(t))\le \gamma^2V(\mbf x(0))\le \gamma^2\norm{v}^2$, we have $\norm{G(\mcl T, \mcl A,\mcl B,\mcl C)}_{ip} \le \gamma$.
        Likewise, satisfaction of \textit{(ii)} implies the norm bound on the dual and the primal PIE response.
\end{proof}


\subsection{Optimal State-feedback Controller Synthesis}
Using the duality results and a Lyapunov function-based characterization of the norm, we can consider the I2P optimal control problem. Specifically, given a PIE of the form,
\begin{align}
    \partial_t(\mcl T\mbf x)(t) &= \mcl A\mbf x(t) + \mcl B w(t) + \mcl B_2u(t), \quad \mcl T\mbf x(0) = 0,\notag\\
    z(t) &= \mcl C\mbf x(t) + \mcl D u(t), \label{eq:PIEcontrol}
\end{align}
the goal is to design a state-feedback $u(t) = \mcl K\mbf x(t)$ such that impulse-input to peak-output norm is optimal. 

Substituting the control policy, the closed-loop system takes the form $G(\mcl T, \mcl A+\mcl B_2\mcl K,\mcl B, \mcl C+\mcl D\mcl K)$ where $G$ is defined in Definition \ref{def:ip-normsystem}. Thus, the control task can be stated as: Given a closed-loop system $G$, find $\mcl K$ minimizing $\gamma$ such that $\sup_{\norm{v}=1} \norm{z}_\infty\le \gamma$.  

Note the auxiliary system admits the form
\begin{align}\label{eqn:PIEauxwinput}
\partial_t(\mcl T\mbf x)(t) &= (\mcl A+\mcl B_2\mcl K)\mbf x(t), \quad\mcl T\mbf x(0) = \mcl Bv, \notag\\
z(t) &= (\mcl C+\mcl D\mcl K)\mbf x(t).
\end{align}



If we apply Thm. \ref{thm:Normcoercive} to the auxiliary system defined above, the constraints in \textit{(i)} are bilinear in $\mcl Q$ and $\mcl K$ and intractable. However, contingent on the change of variable $\mcl Z=\mcl K \mcl Q$, the constraints in \textit{(ii)} are linear and convex. Since \textit{(i)} and \textit{(ii)} are equivalent, we can solve \textit{(ii)} to search for a controller $\mcl K = \mcl Z\mcl Q^{-1}$ that achieves optimal $\gamma$ for the closed-loop system.


\begin{corollary}\label{cor:control}
    Given 4-PI operators $\{\mathcal{T}, \mathcal{A},\mathcal{B}, \mcl B_2,\mathcal{C}, \mcl D\}$, let $\{\mbf x, z\}$ satisfy a PIE of the form Eq. (\ref{eqn:PIEauxwinput}).
    Suppose there exists a $\gamma > 0$
    such that one of the following sets of conditions hold:
    \begin{enumerate}[leftmargin=6mm]
        \item[(i)] $\mcl Q\succ 0$, $\mcl A\mcl Q \mcl T^*+\mcl B_2 \mcl Z \mcl T^* + \mcl T \mcl Q\mcl A^* + \mcl T \mcl Z^* \mcl B_2^* \preceq 0,\vspace{1mm}\\  \begin{bmatrix}
            \gamma^2I&  \mcl B^*\\ \mcl B & \mcl T\mcl Q\mcl T^*
        \end{bmatrix}\!\succeq\! 0,\;\bmat{I & (\mcl C\mcl Q+\mcl D\mcl Z)\\(\mcl C\mcl Q+\mcl D\mcl Z)^*& \mcl Q}\!\succeq\! 0$. \vspace{3mm}
       
        \item[(ii)] $\mcl Q^*\mcl T^*=\mcl T\mcl Q$, $\mcl A \mcl Q +\mcl B_2 \mcl Z +\mcl Q^* \mcl A^* + \mcl Z^* \mcl B_2^* \preceq 0,\vspace{1mm}\\ \begin{bmatrix}
            \gamma^2I & \mcl B^*\\ \mcl B & \mcl Q^* \mcl  T^* 
        \end{bmatrix}\!\succeq\! 0, \;\begin{bmatrix}
            I & (\mcl C\mcl Q+\mcl D\mcl Z)\\  (\mcl C\mcl Q+\mcl D\mcl Z)^*& \mcl Q^*\mcl T^*
        \end{bmatrix}\!\succeq\! 0$. 
    \end{enumerate}
    Then, under the feedback control $u(t)=\mcl K\mbf x(t)$ where $\mcl K= \mcl Z\mcl Q^{-1}$, we have $\sup_{\norm{v}=1} \norm{z}_\infty \le \gamma$.
\end{corollary}
\begin{proof}
The proof follows the exact same steps as proofs of \textit{(ii)} in Thms. \ref{thm:Normcoercive} and \ref{thm:Normnoncoercive}, with the additional change of variable $\mcl K\mcl Q = \mcl Z$. If there exists a $\gamma>0, \mcl Q$ satisfying \textit{(i)} or \textit{(ii)} above, and the solution $\{\mbf x, z\}$ satisfies the closed-loop PIE under the feedback $u(t)=\mcl K\mbf x(t)$, then $\gamma, \mcl Q$ satisfy the conditions in \textit{(ii)} of Thm. \ref{thm:Normcoercive} or \ref{thm:Normnoncoercive}, respectively, for the closed-loop PIE. Thus, the closed-loop PIE Eq. \eqref{eqn:PIEauxwinput} satisfies the norm bound.









\end{proof}

\section{Numerical Results}

In this section, we present numerical examples validating the above theoretical results. First, we present two PDE examples for which I2P norm bounds are computed using the LPIs described in Theorem \ref{thm:Normnoncoercive}. For the state feedback control design, we use the conditions in \textit{(i)} of Cor. \ref{cor:control}. All the steps required to solve the optimization problems in Thms. \ref{thm:Normnoncoercive} and Cor. \ref{cor:control}, are implemented using the PIETOOLS toolbox.

\subsection{Impulse-to-peak Norm Upper-bounding}
\begin{example}

Consider the following transport PDE:
\begin{align*}
  \partial_t x(t,s) &= \partial_sx(t,s)+(s-s^2)w(t),\; z(t) = \int_0^1 x(t,s)\ ds
\end{align*}
where $z$ is the regulated output and Dirichlet boundary condition $x(t,1) = 0$.
For this PDE, we use PIETOOLS to obtain the PIE parameters and solve the optimization problems presented in Thm. \ref{thm:Normnoncoercive} with $\gamma^2$ as the objective to be minimized. 
The upper bounds on I2P norm obtained from the primal and dual LPI solutions is $\gamma = 0.1684$ and $\gamma = 0.1667$ respectively, whereas the theoretical bound is $1/6$.
\end{example}

\begin{example}
Next, we consider the heat equation
\begin{align*}
\partial_t x(t,s) &= \partial^2_sx(t,s)+sw(t),\quad
z(t) = \int_0^1 x(t,s) \ ds,
\end{align*}
with boundary conditions $x(t,0) = \partial_s x(t,1) = 0$. We again solve the LPI problems in Thm. \ref{thm:Normnoncoercive} minimizing $\gamma^2$ and obtain $\gamma = 0.5000$, which matches the theoretical bound $1/2$. 
\end{example}

\subsection{Impulse-to-peak Norm Optimal State Feedback Control}



The first example considers I2P optimal control and stabilization of an unstable reaction-diffusion equation. In the second example, we consider a beam example and design a controller to reduce the peak output norm.

\begin{example}[Unstable Reaction-Diffusion Equation]
Consider an unstable reaction-diffusion equation:
\begin{align*}
    \partial_tx(t,s) &= 14 x(t,s)+ \partial_{s}^2x(t,s) + (s^2-2s)w(t)+u(t),\\
    z(t) &=  \int_0^1 2x(t,s)\ ds,\;\;\; x(t, 0) = \partial_s x(t,1) = 0.
\end{align*}
Since this PDE is unconditionally unstable, the I2P gain is infinity. 
To find an optimal stabilizing control, we solve the LPI presented in $\textit{(i)}$ of Cor. \ref{cor:control} for fixed $\gamma$ and apply bisection method to obtain the best achievable $\gamma$ and the corresponding controller. In this case, we can find a stabilizing controller that achieves $\gamma =  1.375$. The closed-loop system response with the obtained controller is shown in Fig. \ref{fig:reactevol}. Additionally, the regulated output is shown in Fig. \ref{fig:reactreg}.
\end{example}

\begin{figure}
    \centering
    \includegraphics[width=0.99\linewidth]{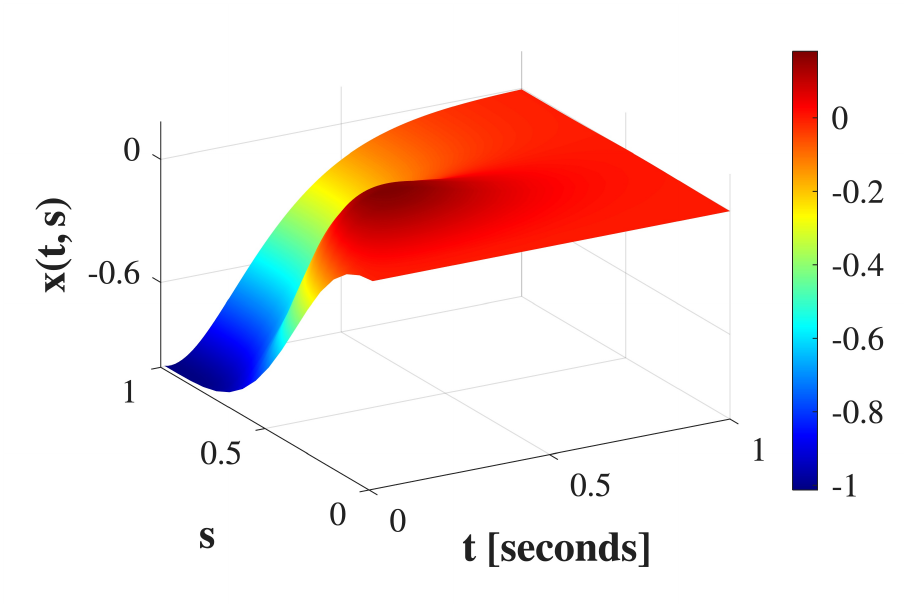}
    \caption{Evolution of the PDE state for the reaction-diffusion equation under state-feedback controller}
    \label{fig:reactevol}
    \vspace{-15pt}
\end{figure}
\begin{figure}
    \centering
    \includegraphics[width=0.75\linewidth]{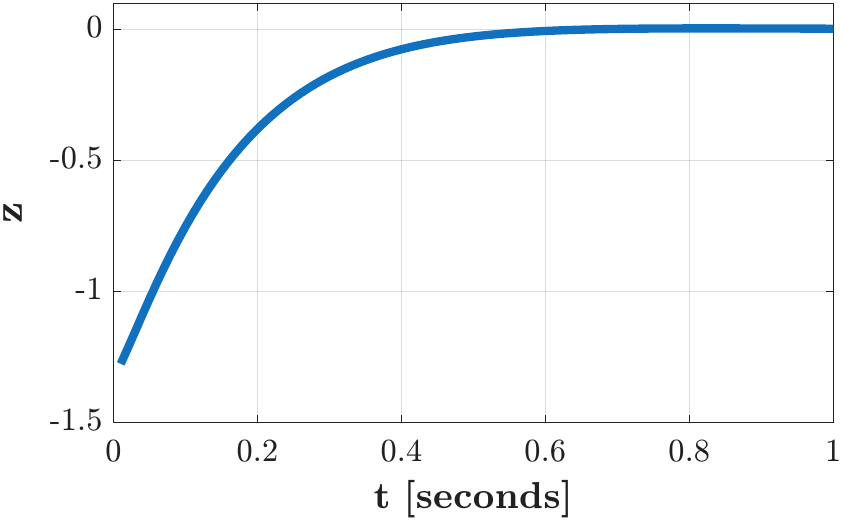}
    \caption{The regulated output, $\int_0^12 x(t,s) ds$, for the reaction-diffusion equation. The control stabilizes the system and satisfies peak output bounds.}
    \label{fig:reactreg}
\end{figure}

\begin{example}[Timoshenko Beam Equations]
Considering a beam with transverse displacement $w$ and cross-section rotation angle $\phi$, the Timoshenko beam equations model the evolution of these quantities with the following equations:
\begin{align}
\ddot{w} = \partial_s(w_s - \phi)+u+sw, \;\; \ddot{\phi} = \phi_{ss} +(w_s - \phi)  \nonumber 
\\ \phi(0) = w(0) = 0, \; \phi_s(0) = 0, \; w_s(1) - \phi(1) = 0. \nonumber
\end{align}

Then, we can choose $\mbf x = \begin{bmatrix}
    \dot{w},  w_s -\phi, \dot{\phi}, \phi_s
\end{bmatrix}$ and get:
\begin{align}
    \dot{\mbf x}(t,s) = \begin{bmatrix}
        0 & 0&0&0\\ 0&0&-1&0\\ 0 &1 &0&0\\ 0&0&0&0
    \end{bmatrix} \mathbf{x}(t,s) + \begin{bmatrix}
        0&1&0&0\\ 1&0&0&0\\ 0&0&0&1\\ 0&0&1&0
    \end{bmatrix} \partial_s \mbf{x}(t,s)\nonumber\\+\begin{bmatrix}
        1&0&0&0
    \end{bmatrix}^\top u + \begin{bmatrix}
        s&0&0&0
    \end{bmatrix}^\top w\nonumber,    
    \\\mathbf{x}_1(t,0) = \mathbf{x}_3(t,0) = \partial_s\mathbf{x}_4(t,0) = \partial_s\mathbf{x}_2(t,0) = 0 \nonumber
\end{align} with output  $z(t) = \int_0^1 \begin{bmatrix}
        1 & 0&0&0
    \end{bmatrix} \mathbf{x}(t,s) ds$. 

The closed-loop I2P-norm is 0.50, while the open-loop I2P-norm is 0.54, proving that controller suppresses peak-output norm.
\end{example}
\begin{figure}
    \centering
    \includegraphics[width=0.75\linewidth]{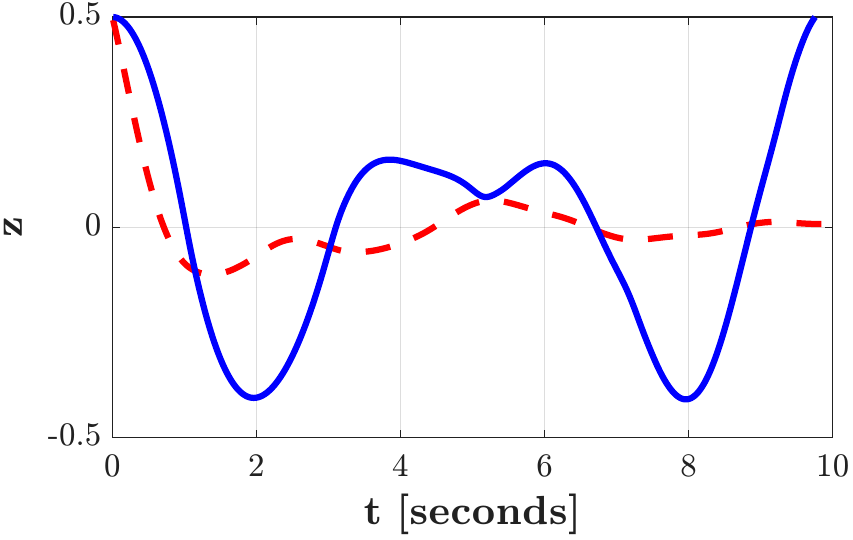}
    \caption{The regulated output, $\int_0^1 \mathbf{x}_1(t,s)ds$, for the controlled ($\textcolor{red}{--}$) and uncontrolled ($\textcolor{blue}{-}$) Timoshenko beam equations.}
    \vspace{-14pt}
    \label{fig:placeholder}
\end{figure}

\begin{figure}
    \centering
    \vspace{-15pt}
    \includegraphics[width=0.99\linewidth]{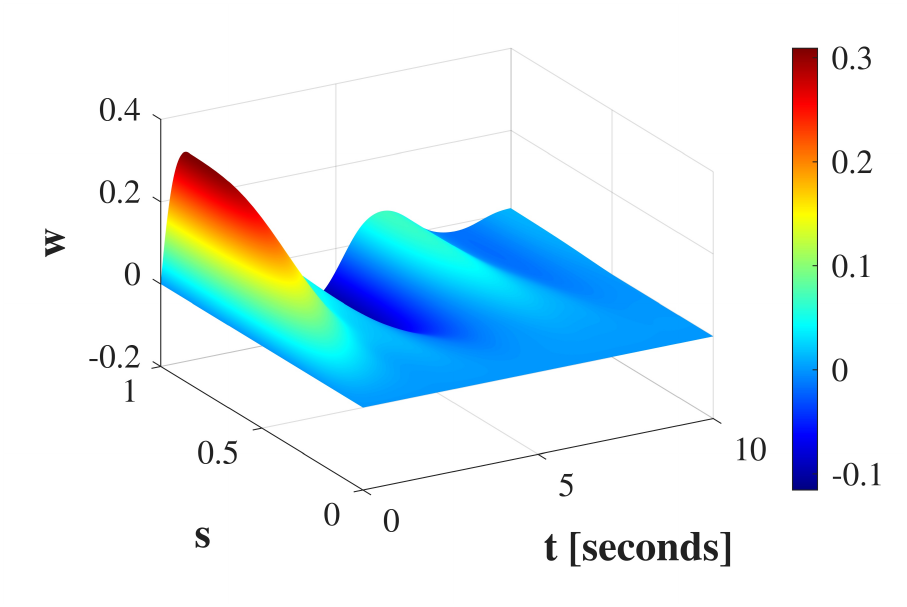}
    \caption{Evolution of ${w}$, for the Timoshenko Beam Equations under state-feedback control.}
    \label{fig:placeholder}
    \vspace{-15pt}
\end{figure}

\subsection{A Note on Computation}

Table \ref{tab:execution_time_comparison} provides the time per bisection iteration for the state-feedback control design problems we considered. It also displays the size of the underlying LMIs.

\begin{table}[h] 
\centering 
\begin{tabular}{lcccc} 
\textbf{\makecell{Model\\Type}} & \textbf{\makecell{Time \\(in CPU sec.)}} & \textbf{\makecell{\# of \\iterations}}& \textbf{\makecell{Decision \\Variables}} & \textbf{Constraints} \\ 
\midrule
Beam & 14.59 & 15 & 203,065 & 4,353 \\ 
\makecell{Reaction-\\diffusion} & 11.89 & 5 & 183,840& 1,054\\
\end{tabular} 
\caption{Computation Parameter and Time Comparison}
\vspace{-20pt}
\label{tab:execution_time_comparison} 
\end{table}

\section{Conclusion}
Despite its relevance in safe operations, a computational method to characterize I2P norm for a sufficiently general class of PDEs was not available. In this work, we leveraged the PIE representation to develop Lyapunov-based convex optimization formulation of the I2P norm bounding problem for arbitrary linear PDEs. We showed the bounds so obtained are provable in the sense the feasibility of LPIs provide a certificate for the I2P-norm bound. We also demonstrated there is no theoretical conservatism in the primal and dual I2P-norms. We solved the impulse-to-peak norm upper-bounding and norm-optimal state-feedback control synthesis problem for linear PDEs and validated the solution for various numerical examples, for which analytical bounds are known.

\bibliographystyle{IEEEtran} 
\bibliography{name}  

@article{TOKUNAGA19981473,
title = {{Analysis and Synthesis of the Robust Impulse-to-Peak Performance}},
journal = {Automatica},
volume = {34},
number = {11},
pages = {1473-1477},
year = {1998},
issn = {0005-1098},
doi = {https://doi.org/10.1016/S0005-1098(98)00096-X},
author = {H. Tokunaga and T. Iwasaki and S. Hara},
}

@article{shivakumar_extension_2024,
    title = {{Extension of the Partial Integral Equation Representation to {GPDE} Input-Output Systems}},
    urldate = {2025-10-20},
    journal = {IEEE Transactions on Automatic Control},
    author = {S. Shivakumar and A. Das and S. Weiland and M. Peet},
    volume={70},
    number={5},
    pages={3240-3255},
    year = {2024},
}

@article{shivakumar_dual_2024,
  title={Dual Representations and ${H}_{\infty}$-Optimal Control of Partial Differential Equations},
  author={Shivakumar, Sachin and Das, Amritam and Peet, Matthew},
  journal={IEEE Transactions on Automatic Control},
  year={2026},
  publisher={IEEE}
}

@article{shivakumar_pietools_2025,
    title = {{PIETOOLS} 2024: {User} {Manual}},
    author = {S. Shivakumar and D. Jagt and D. Braghini and A. Das and Y. Peet and M. Peet},
    year = {2025},
    journal = {arXiv preprint arXiv:2501.17854},
}

@ARTICLE{balas,
  author={Balas, M.},
  journal={IEEE Transactions on Automatic Control}, 
  title={{Feedback Control of Flexible Systems}}, 
  year={1978},
  volume={23},
  number={4},
  pages={673-679},
  doi={10.1109/TAC.1978.1101798}}

@book{Lasiecka_Triggiani_2000, place={Cambridge}, series={Encyclopedia of Mathematics and its Applications}, title={Control Theory for Partial Differential Equations: Continuous and Approximation Theories}, publisher={Cambridge University Press}, author={I. Lasiecka and R. Triggiani}, year={2000}, collection={Encyclopedia of Mathematics and its Applications}}

@article{Moura,
    author = {S. J. Moura and H. K. Fathy},
    title = {{Optimal Boundary Control of Reaction–Diffusion Partial Differential Equations via Weak Variations}},
    journal = {Journal of Dynamic Systems, Measurement, and Control},
    volume = {135},
    number = {3},
    pages = {034501},
    year = {2013},
    issn = {0022-0434},
    doi = {10.1115/1.4023071},
    eprint = {https://asmedigitalcollection.asme.org/dynamicsystems/article-pdf/135/3/034501/6113724/ds_135_3_034501.pdf},
}

@article{vanKeulen,
author = {B. van Keulen},
title = {{Redheffer’s Lemma and $\mathcal{H}_{\infty}$-Control for Infinite-Dimensional Systems}},
journal = {SIAM Journal on Control and Optimization},
volume = {32},
number = {1},
pages = {261-278},
year = {1994},
doi = {10.1137/S036301299121741X}
}

@article{mouraCST,
  author={S. Dey and H. E. Perez and S. J. Moura},
  journal={IEEE Transactions on Control Systems Technology}, 
  title={{Model-Based Battery Thermal Fault Diagnostics: Algorithms, Analysis, and Experiments}}, 
  year={2019},
  volume={27},
  number={2},
  pages={576-587},
  doi={10.1109/TCST.2017.2776218}}

@article{Guay,
author = {H. Shang and J. Forbes and M. Guay},
title = {{Model Predictive Control for Quasilinear Hyperbolic Distributed Parameter Systems}},
journal = {Industrial \& Engineering Chemistry Research},
volume = {43},
number = {9},
pages = {2140-2149},
year = {2004},
doi = {10.1021/ie030653z},
eprint = {https://doi.org/10.1021/ie030653z}}

@book{curtain2012introduction,
  title={An Introduction to Infinite-Dimensional Linear Systems Theory},
  author={R. F. Curtain and H. Zwart},
  volume={21},
  year={2012},
  publisher={Springer Science \& Business Media}
}

@article{PEET2021109473,
title = {{A Partial Integral Equation ({PIE}) Representation of Coupled Linear {PDEs} and Scalable Stability Analysis Using {LMIs}}},
journal = {Automatica},
volume = {125},
pages = {109473},
year = {2021},
issn = {0005-1098},
doi = {https://doi.org/10.1016/j.automatica.2020.109473},
author = {Matthew Peet},}

@book{troltzsch2010optimal,
  title={Optimal Control of Partial Differential Equations: Theory, Methods, and Applications},
  author={F. Tr{\"o}ltzsch},
  volume={112},
  year={2010},
  publisher={American Mathematical Society}
}

@article{mironchenko2020input,
  title={{Input-to-State Stability of Infinite-Dimensional Systems: Recent Results and Open Questions}},
  author={A. Mironchenko and C. Prieur},
  journal={SIAM Review},
  volume={62},
  number={3},
  pages={529--614},
  year={2020},
  publisher={SIAM}
}

\section{Appendix}
\subsection{Additional Examples}\
\begin{example} (More Beam Plots)
Additional plots showing the evolution of actual PDE states for the Timoshenko Beam example are given in Fig. \ref{fig:placeholder1} and \ref{fig:placeholder2}.
\begin{figure}
    \centering
    \includegraphics[width=0.99\linewidth]{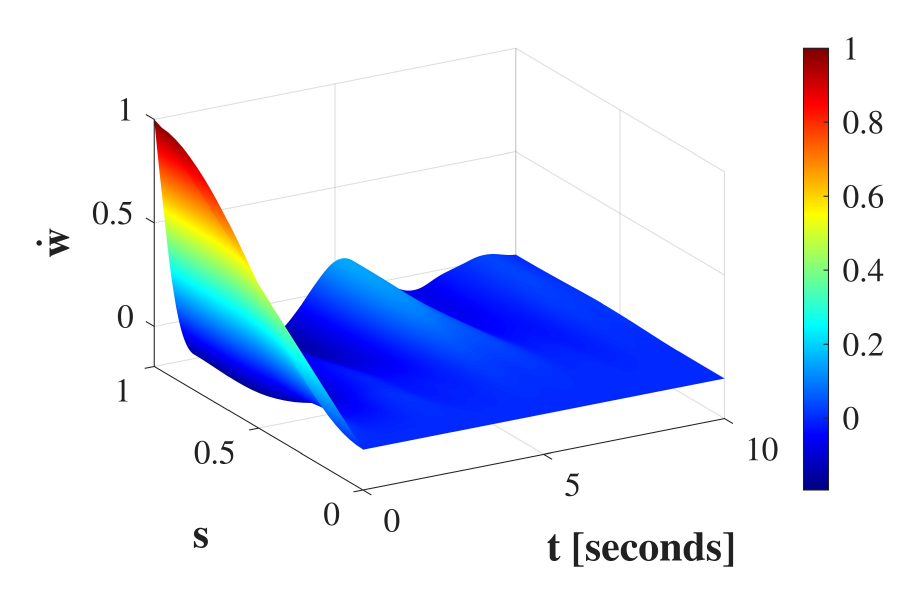}
    \caption{Evolution of the first PDE state, $\dot{w}$, for the Timoshenko Beam Equations under state-feedback control.}
    \label{fig:placeholder1}
\end{figure}

\begin{figure}
    \centering
    \includegraphics[width=0.99\linewidth]{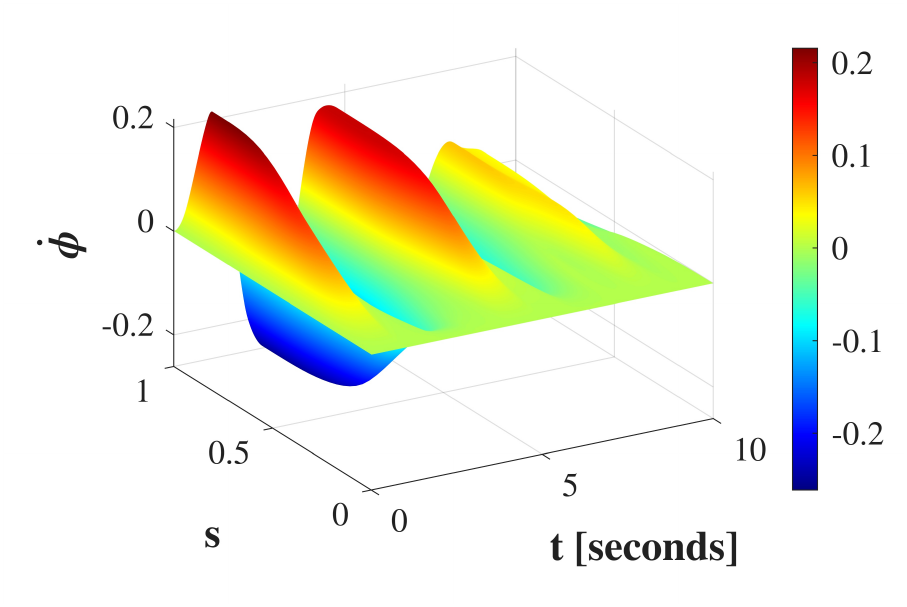}
    \caption{Evolution of the first PDE state, $\dot{\phi}$, for the Timoshenko Beam Equations under state-feedback control.}
    \label{fig:placeholder2}
\end{figure}
\end{example}
\begin{example}[Transport Equation]
Next, we consider a transport equation with control as
\begin{align*}
    \partial_tx(t,s) &= \partial_{s}x(t,s) + 10s(s-1)(s-0.5)w(t) +u(t),\\
    z(t) &=  \int_0^1 x(t,s)\ ds,\quad x(t, 1) = 0.
\end{align*}
Although the transport equation is naturally stable, the control action can be utilized to suppress vibrations. Similar to unstable reaction-diffusion PDE, we search for an optimal controller using LPIs in Cor. \ref{cor:control} by bisecting on the $\gamma$ parameter. As seen in Fig. \ref{fig:transportresponse}, we see the stability is maintained. Moreover, the uncontrolled system has an impulse-to-peak bound of $0.1678$, whereas with control we can achieve an impulse-to-peak norm of $0.1247$ (see Fig. \ref{fig:controlledregulatedoutputstransport}). The transport problem had a mean iteration time in CPU seconds of 3.43 over 15 iterations with a problem size involing 85271 decision variables and 699 constraints. 
\end{example}

\begin{figure}
    \centering
    \includegraphics[width=0.99\linewidth]{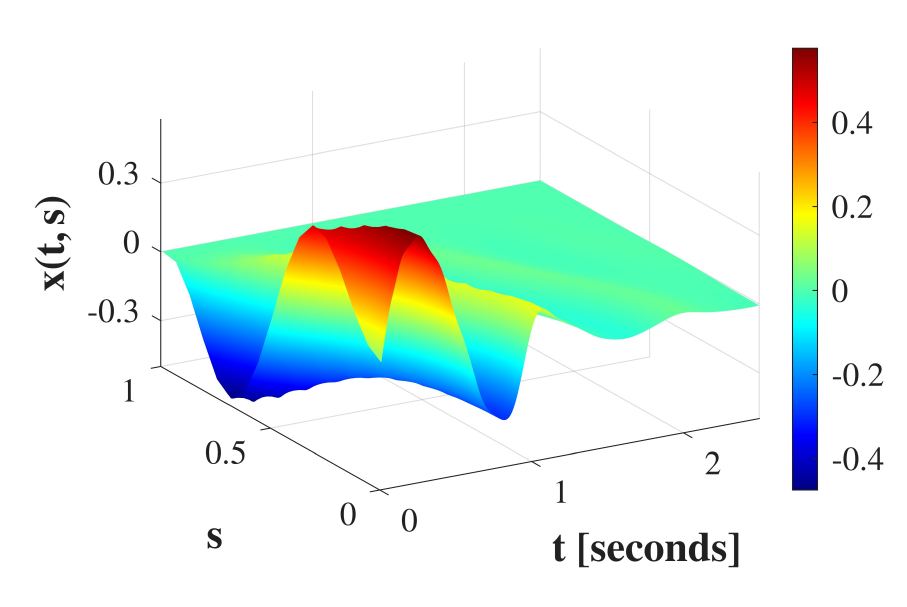}
    \caption{Evolution of the PDE state for the transport equation under state-feedback controller}
    \label{fig:transportresponse}
\end{figure}
\begin{figure}
    \centering
    \includegraphics[width=0.8\linewidth]{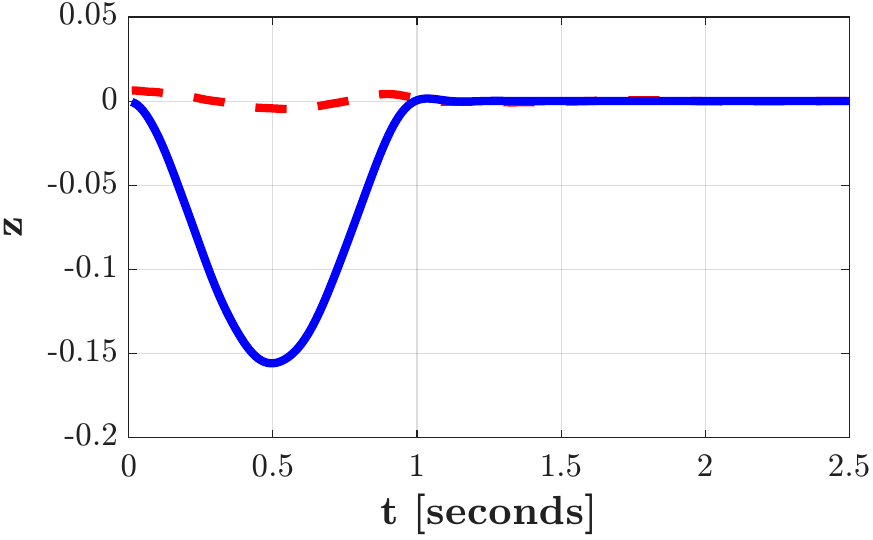}
    \caption{The regulated output, $\int_0^1 x(t,s)ds$, for the controlled ($\textcolor{red}{--}$) and uncontrolled ($\textcolor{blue}{-}$) transport equation. The control reduces the peak response under impulsive input in comparison to uncontrolled system.}
    \label{fig:controlledregulatedoutputstransport}
\end{figure}

\end{document}